\documentclass[11pt,reqno]{amsart}

\usepackage{amsfonts, amsthm, amsmath}
\allowdisplaybreaks[4]

\usepackage{rotating}

\usepackage{tikz}

\usetikzlibrary{decorations.markings} 

\usepackage{graphics}

\usepackage{amssymb}

\usepackage{amscd}

\usepackage[latin2]{inputenc}

\usepackage{t1enc}

\usepackage[mathscr]{eucal}

\usepackage{indentfirst}

\usepackage{graphicx}

\usepackage{graphics}

\usepackage{pict2e}

\usepackage{mathrsfs}

\usepackage{enumerate}
\usepackage[pagebackref]{hyperref}
\hypersetup{colorlinks=true}
\usepackage{cite}
\usepackage{color}
\usepackage{epic}
\usepackage{hyperref} 
\usepackage{framed}
\usepackage{mathabx}
\usepackage{scalerel}
\usepackage{booktabs}
\usepackage{makecell}
\newcolumntype{V}{!{\vrule width 2pt}}

\numberwithin{equation}{section}
\def\blue{\textcolor{blue}}

\theoremstyle{plain}

\newtheorem{theorem}{Theorem}[section]

\newtheorem{lemma}[theorem]{Lemma}

\newtheorem{corollary}[theorem]{Corollary}

\newtheorem{Fact}[theorem]{Fact}

\theoremstyle{definition}

\newtheorem{Def}[theorem]{Definition}

\newtheorem{example}[theorem]{Example}

\newtheorem{remark}{Remark}

\newtheorem{?}[theorem]{Problem}

\newcommand{\N}{\mathbb{N}}
\newcommand{\Z}{\mathbb{Z}}
\newcommand{\s}{\mathfrak{S}}
\newcommand{\Q}{\mathcal{Q}}
\newcommand{\bQ}{\overline{\mathcal{Q}}}
\newcommand{\A}{\mathcal{A}}
\newcommand{\M}{\mathcal{M}}
\newcommand{\T}{\mathcal{T}}
\newcommand{\R}{\mathcal{R}}
\newcommand{\B}{\mathcal{B}}
\newcommand{\G}{\mathcal{G}}
\newcommand{\UT}{\mathcal{UT}}
\newcommand{\UB}{\mathcal{UB}}
\newcommand{\UG}{\mathcal{UG}}
\def\red{\mathrm{red}}

\def\des{\mathrm{des}}
\def\asc{\mathrm{asc}}
\def\cdes{\mathrm{cdes}}
\def\casc{\mathrm{casc}}
\def\plat{\mathrm{plat}}
\def\lleaf{\mathrm{leaf}^*}
\def\ddes{\mathrm{dd}}
\def\sdes{\mathrm{sd}}
\def\sddes{\mathrm{dsd}}
\def\emp{\mathrm{emp}}
\def\bar{\mathrm{bar}}
\def\dcdes{\mathrm{dcdes}}

\begin{document}

\title[Rooted quasi-Stirling permutations of general multisets]{Rooted quasi-Stirling permutations of general multisets}

\author[S. Fu]{Shishuo Fu}
\address[Shishuo Fu]{College of Mathematics and Statistics, Chongqing University, Huxi campus, Chongqing 401331, P.R. China}
\email{fsshuo@cqu.edu.cn}

\author[Y. Li]{Yanlin Li}
\address[Yanlin Li]{College of Mathematics and Statistics, Chongqing University, Huxi campus, Chongqing 401331, P.R. China}
\email{lly.1997112@foxmail.com}

\date{\today}

\begin{abstract}
Given a general multiset $\M=\{1^{m_1},2^{m_2},\ldots,n^{m_n}\}$, where $i$ appears $m_i$ times, a multipermutation $\pi$ of $\M$ is called {\em quasi-Stirling}, if it contains no subword of the form $abab$ with $a\neq b$. We designate exactly one entry of $\pi$, say $k\in \M$, which is not the leftmost entry among all entries with the same value, by underlining it in $\pi$, and we refer to the pair $(\pi,k)$ as a quasi-Stirling multipermutation of $\M$ rooted at $k$. By introducing certain vertex and edge labeled trees, we give a new bijective proof of an identity due to Yan, Yang, Huang and Zhu, which links the enumerator of rooted quasi-Stirling multipermutations by the numbers of ascents, descents, and plateaus, with the exponential generating function of the {\em bivariate Eulerian polynomials}. This identity can be viewed as a natural extension of Elizalde's result on $k$-quasi-Stirling permutations, and our bijective approach to proving it enables us to
\begin{itemize}
	\item prove bijectively a Carlitz type identity involving quasi-Stirling polynomials on multisets that was first obtained by Yan and Zhu.
	\item confirm a recent partial $\gamma$-positivity conjecture due to Lin, Ma and Zhang, and find a combinatorial interpretation of the $\gamma$-coefficients in terms of two new statistics defined on quasi-Stirling multipermutations called sibling descents and double sibling descents.
\end{itemize}
\end{abstract}



\maketitle


\section{Introduction}\label{sect: intro}

Let $\s_n$ denote the set of permutations of $[n]:=\{1,2,\ldots,n\}$. For any $\pi=\pi_1\pi_2\cdots\pi_n\in\s_n$, we call position $i$, $1\le i\le n$, a {\em descent} of $\pi$, provided that $\pi_i>\pi_{i+1}$, where $\pi_{n+1}=\pi_0=0$, a convention that we follow in this paper. Denote $\des(\pi)$ the total number of descents of $\pi$. The polynomials $A_0(t)=1$, and
\begin{align}\label{comb:Eulerian poly}
A_n(t):=\sum_{\pi\in\s_n}t^{\des(\pi)}=\sum_{i=1}^{n}A_{n,i}t^i,
\end{align}
for $n\ge 1$, are the well known {\em Eulerian polynomials}, whose coefficient $A_{n,i}$ is called the {\em Eulerian number} and gives the number of permutations $\pi\in\s_n$ having exactly $i$ descents. The following relation is called the {\em Carlitz identity} in Petersen's book \cite{Petbook} and was known to Euler (see e.g.~\cite{Foa}). It can be used as an alternative definition of Eulerian polynomials.
\begin{align}\label{numerator:Eulerian poly}
\sum_{m=0}^{\infty}m^n t^m = \frac{A_n(t)}{(1-t)^{n+1}}.
\end{align}
Euler also derived the exponential generating function
\begin{align}\label{gf:Eulerian poly}
A(t,u):=\sum_{n=0}^{\infty}A_n(t)\frac{u^n}{n!}=\frac{1-t}{1-te^{(1-t)u}}.
\end{align}

When $m^n$ is replaced by the Stirling number of the second kind, say $S(m+n,m)$, in \eqref{numerator:Eulerian poly}, a formula analogous to \eqref{numerator:Eulerian poly} can be derived as
\begin{align}\label{numerator:Stirling perm}
\sum_{m=0}^{\infty}S(m+n,m)t^m=\frac{Q_n(t)}{(1-t)^{2n+1}},
\end{align} 
whose numerator polynomial $Q_n(t)$ together with its combinatorial interpretation, were introduced and studied by Gessel and Stanley \cite{GS} via the notion of \emph{Stirling permutations}. Namely, the set of Stirling permutations of order $n$, denoted as $\Q_n$, are the collection of permutations $\pi=\pi_1\pi_2\cdots\pi_{2n}$ of the multiset $\{1^2,2^2,\ldots,n^2\}$, subject to the condition that if $i<j<k$ and $\pi_i=\pi_k$, then we must have $\pi_j>\pi_i$. Two words $u=u_1\cdots u_m$ and $v=v_1\cdots v_m$ are said to be {\em order isomorphic}, if we have $u_i>u_j$ (resp.~$u_i=u_j$, $u_i<u_j$) if and only if $v_i>v_j$ (resp.~$v_i=v_j$, $v_i<v_j$) for all $1\le i<j\le m$. Using this notion of order isomorphism, $\pi$ is a Stirling permutation precisely when it contains no subword that is order isomorphic to $212$. $Q_n(t)$ can now be interpreted as the generating function over $\Q_n$ by the number of descents, and it will be referred to as the {\em Stirling polynomial} in the sequel.

There is a natural bijection, the so-called {\em Koganov-Janson correspondence} \cite{Kog,Jan,Eli}, between Stirling permutations and labeled increasing plane trees. In a recent work \cite{AGPS}, Archer et al. considered lifting the increasing restriction on the tree side and finding the counterpart on the permutation side via the aforementioned correspondence. They called this bigger set of permutations the {\em quasi-Stirling permutations}, which are permutations $\pi$ of $\{1^2,2^2,\ldots,n^2\}$ such that there exists no subword of $\pi$ that is order isomophic to $1212$ or $2121$. We denote the set of quasi-Stirling permutations of order $n$ as $\bQ_n$, and let $\overline{Q}_n(t)=\sum_{\pi\in\bQ_n}t^{\des(\pi)}$ be its descent polynomial (called quasi-Stirling polynomial in what follows), where for $\des(\pi)$ we assume the same convention $\pi_0=\pi_{2n+1}=0$.

Motivated by those classical results in the literature for Eulerian polynomials and Stirling polynomials, Elizalde successfully developed in his recent work \cite{Eli} several parallel or new results for quasi-Stirling polynomials, such as

\begin{align}\label{qStir-Euler}
(n+1)\overline{Q}_n(t) &= n![u^n]A(t,u)^{n+1},\; \text{and}\\
\label{numerator:qStir}
\sum_{m=0}^{\infty}m^n\binom{m+n}{m}t^m &= \frac{(n+1)\overline{Q}_n(t)}{(1-t)^{2n+1}}.
\end{align}

\begin{remark}
Note that in Elizalde's original formulation, the factor $(n+1)$ has been divided from both sides of equations \eqref{qStir-Euler} and \eqref{numerator:qStir}. Putting them in the present form, we are naturally led to consider the notion of rooted quasi-Stirling permutations, which we will introduce in the next section.
\end{remark}

Elizalde derived \eqref{qStir-Euler} by first establishing an implicit equation satisfied by the generating functions and then extracting the coefficients. He next utilized \eqref{qStir-Euler} to deduce \eqref{numerator:qStir}, a nice analogue for quasi-Stirling polynomials $\overline{Q}_n(t)$ of \eqref{numerator:Eulerian poly} for $A_n(t)$ and \eqref{numerator:Stirling perm} for $Q_n(t)$. He ended his paper \cite{Eli} by raising the problem of giving a combinatorial proof of \eqref{numerator:qStir} that is reminiscent of Gessel and Stanley's second proof of \eqref{numerator:Stirling perm}. This in turn
has motivated Yan and her collaborators to work out three papers \cite{YZ,YYHZ,YHY}. Among the results derived by Yan et al., we would like to highlight the following two identities \eqref{genqStir-Euler} and \eqref{numerator:genqStir}. Some definitions are needed to state these results.

Let $\A$, a subset of $\Z_{>0}$, be our alphabet. For a word $w=w_1\cdots w_n\in\A^n$, an index $i$, $0\le i\le n$, is an {\em ascent} (resp.~a {\em plateau}) of $w$ if $w_i<w_{i+1}$ (resp.~$w_i=w_{i+1}$), where we use the same convention that $w_0=w_{n+1}=0$. In particular, the empty word $\epsilon$ has one plateau coming from the initial and final $0$s that we have appended to $\epsilon$ by convention. The number of ascents (resp.~plateaus) of $w$ will be denoted as $\asc(w)$ (resp.~$\plat(w)$). 

Note that two order isomorphic words are indistinguishable, when we enumerate them with respect to various statistics such as $\des,~\asc,~\plat$, etc. It is convenient to introduce the {\em reduction map} ``$\red$''. Namely, for any word $w$ consisted of integers, we obtain the unique word $\red(w)$ of the same length by replacing the $i$-th smallest letter in $w$ by $i$. For instance, $\red(31355)=21233$. It is evident that two words $w$ and $v$ are order isomorphic, if and only if $\red(w)=\red(v)$.

The notion of quasi-Stirling permutations can be extended to any multiset $\M=\{1^{m_1},\ldots,n^{m_n}\}$. In view of the reduction map, we always assume without the loss of generality that each $m_i\ge 1$. Denote by $\bQ_{\M}$ the set of all quasi-Stirling permutations of $\M$. 
We shall consider the trivariate enumerator 

$$\overline{Q}_{\M}(x,y,z)=\sum_{\pi\in\bQ_{\M}}x^{\des(\pi)}y^{\asc(\pi)}z^{\plat(\pi)}.$$

Let $A(x,y,u)=\sum_{n\ge 0}A_n(x,y)\frac{u^n}{n!}$ be the generating function of the {\em bivariate Eulerian polynomial}

$$A_n(x,y)=\sum_{\pi\in\s_n}x^{\des(\pi)}y^{\asc(\pi)}.$$

Yan, Yang, Huang and Zhu derived in \cite[Coro.~1.5]{YYHZ} the following identity connecting $\overline{Q}_{\M}(x,y,z)$ with $A(x,y,u)$. 

\begin{theorem}\label{thm:YYHZ}
Let $\M=\{1^{m_1},2^{m_2},\ldots,n^{m_n}\}$ with $M=m_1+\cdots+m_n$. We have
\begin{align}\label{genqStir-Euler}
& (M-n+1)\overline{Q}_{\M}(x,y,z)=n![u^n](A(x,y,u)-1+z)^{M-n+1}.
\end{align}
\end{theorem}

Note that setting $m_1=m_2=\cdots=m_n=k$ in \eqref{genqStir-Euler} recovers Elizalde's result \cite[Eq.~(22)]{Eli} for $k$-quasi-Stirling permutations, which further reduces to \eqref{qStir-Euler} in the case of $k=2$. Moreover, we remark again that the original form of \eqref{genqStir-Euler} in \cite{YYHZ} (as well as the form of \eqref{numerator:genqStir} below in \cite{YZ}) divides the factor $M-n+1$ from both sides of the equation. This distinction, albeit cosmetic when viewed algebraically, leads us to a completely different combinatorial approach from that of \cite{YYHZ}. Relying on the insight we gained from this new bijective proof of Theorem~\ref{thm:YYHZ}, we are able to give a new bijective proof of the following Carlitz type identity for $\overline{Q}_{\M}(t):=\overline{Q}_{\M}(t,1,1)$, which first appeared as Theorem 1.2 in Yan and Zhu's paper \cite{YZ}.

\begin{theorem}\label{thm:YZ}
Let $\M=\{1^{m_1},2^{m_2},\ldots,n^{m_n}\}$ with $M=m_1+\cdots+m_n$. We have
\begin{align}
\label{numerator:genqStir}
& \sum_{m\ge 0}\binom{M-n+m}{m}m^n t^m = \frac{(M-n+1)\overline{Q}_{\M}(t)}{(1-t)^{M+1}}.
\end{align}
\end{theorem}

The third main result of this paper, which also follows from our proof of \eqref{genqStir-Euler}, is the following partial $\gamma$-positive expansion for $\overline{Q}_{\M}(x,y,z)$. 

\begin{theorem}\label{thm:pargamma}
For any multiset $\M=\{1^{m_1},2^{m_2},\ldots,n^{m_n}\}$ with $M=m_1+\cdots+m_n$, the polynomial $\overline{Q}_{\M}(x,y,z)$ is partial $\gamma$-positive and has the expansion
\begin{align}\label{pargamma of quasiStir}
\overline{Q}_{\M}(x,y,z)=\sum_{i=0}^{M-n}z^i\sum_{j=1}^{\lfloor\frac{M+1-i}{2}\rfloor}\gamma_{\M,i,j}(xy)^j(x+y)^{M+1-i-2j},
\end{align}
where
\begin{align}\label{pargamma-coef}
\gamma_{\M,i,j}=\#\{\pi\in\bQ_{\M}:\plat(\pi)=i,~\sdes(\pi)=j,~\sddes(\pi)=0\}.
\end{align}
\end{theorem}

Note that the nonnegativity of the coefficients $\gamma_{\M,i,j}$ was previously conjectured by Lin, Ma, and Zhang \cite{LMZ}, and was first confirmed by Yan, Huang, and Yang \cite{YHY}. In that same paper Yan et al. also provided a combinatorial interpretation of $\gamma_{\M,i,j}$ that is different from the one we give here in \eqref{pargamma-coef}. The meaning of partial $\gamma$-positivity and the definitions of the {\em sibling descent} and {\em double sibling descent} (denoted respectively as $\sdes$ and $\sddes$ in \eqref{pargamma-coef}) will be introduced in the final section, where Theorem~\ref{thm:pargamma} will be proved as well. 

For the rest of the paper, we first introduce in section \ref{sect:tree} certain vertex and edge labeled trees, as well as the notion of rooted quasi-Stirling permutations of general multisets. These two kinds of combinatorial objects are in natural bijection with each other. Building on this bijection, we present new bijective proofs of Theorems \ref{thm:YYHZ} and \ref{thm:YZ} in section \ref{sect:thm1-2}.

\section{VE-labeled trees}\label{sect:tree}

Recall that the Koganov-Janson correspondence mentioned in the introduction links Stirling permutations with labeled increasing plane trees, where the labels are placed on every edge. On the other hand, Yan et al. utilized certain vertex-labeled plane trees in both of their papers \cite{YYHZ,YZ}. For our purpose, it is convenient to consider certain plane rooted trees where both vertices and edges are labeled. The main goal of this section is to introduce this new tree model and the rooted quasi-Stirling multipermutations. We should remark that the use of this tree model could be bypassed entirely, but we believe that making use of it enhances the readability and makes several terminologies self-explanatory. 

All the trees considered in this paper (ordered or unordered) will be rooted. Each non-root vertex, say $v$, in a tree $T$ has a unique vertex connected to it that is the closest vertex to $v$ on the path from $v$ to the root of $T$. We call this unique vertex the {\em parent} of $v$, and denote it as $p_T(v)$, or simply $p(v)$ when the tree (or the graph) under consideration is clear from the context. $v$ is then called a {\em child} of $p(v)$. Two vertices are called {\em siblings} if they share the same parent, and the two edges connecting them to this parent are said to be {\em sibling edges} of each other as well. Take the tree in Fig.~\ref{fig:VE-tree} for example, the labeled vertices $10$ and $11$ are siblings with the vertex $8$ being their common parent. The vertices $9$ and $s_6$ are also siblings of each other, where the use of $s_6$ will be explained in Definition~\ref{def:VE-tree}. All edges in a tree are thought of as pointing towards the root, so that the edge $\overrightarrow{uv}$ is said to be starting at $u$ and ending at $v$, and the tree itself is viewed as a directed graph. For instance, the edge labeled $2$ in Fig.~\ref{fig:VE-tree} starts at vertex $1$ and ends at vertex $9$. We are now ready to give the first key definition of this paper. Recall that $\chi(S)=1$ if the statement $S$ is true and $\chi(S)=0$ otherwise.

\begin{Def}\label{def:VE-tree}
Given any multiset $\M=\{1^{m_1},\ldots,n^{m_n}\}$ with $M=m_1+\cdots+m_n$, we denote $\T_{\M}$ the set of {\em vertex and edge labeled trees} (abbreviated as {\em VE-labeled trees} in what follows) over $\M$. These are plane rooted trees with $M-\sum_{1\le i\le n}\chi(m_i>1)$ edges that satisfy the following conditions.
\begin{enumerate}
	\item The labels of vertices are all distinct and form precisely the set $[M-n]_0\bigcup S_{\M}$, where $$[M-n]_0:=\{0,1,2,\ldots,M-n\}, \text{ and } S_{\M}:=\{s_i:m_i=1\}.$$
	We use letter $s$ with subscript $i$, so that the label $i$ from $[M-n]_0$ and the singleton $i\in\M$ could be distinguished.
	\item $S_{\M}$ is called the set of {\em singletons} of $\M$. A vertex receives a label $s_i\in S_{\M}$ if and only if it is a leaf which starts an edge that has label $i$.
	\item Every edge receives a unique label form the multiset $\M\setminus\{i\in [n]:m_i>1\}$. Edges with the same label must be adjacent sibling edges.
	\item The integer-labeled vertices and the labels of edges starting at them are {\em compatible} in the following sense. For edges with the same label, their starting vertices are increasingly labeled from left to right. For two edges labeled $e_1$ and $e_2$ ($\neq e_1$) that start at vertices with integer labels $v_1$ and $v_2$ respectively, we must have that $e_1<e_2$ if and only if $v_1<v_2$.
\end{enumerate}
The trees in $\T_{\M}$ whose roots are labeled as $0$ are said to be {\em regular}. They form a subset which we denote as $\T_{\M}^0$.
\end{Def}

The reader is encouraged to use the tree in Fig.~\ref{fig:VE-tree}, whose labels of all the vertices have been colored blue, to check all the conditions in Definition~\ref{def:VE-tree}.

\begin{remark}\label{rmk:conversion}
It should be pointed out that as a consequence of condition (4), once we fix the label of the root, the labeling of all the edges implies uniquely the eligible labeling for the vertices and vice versa. Especially in the case of $\M=\{1^2,2^2,\ldots,n^2\}$, i.e., the original quasi-Stirling permutations as introduced by Archer et al. \cite{AGPS}, there is a one-to-one correspondence between the vertex-labels and edge-labels (although an obvious shift of values is needed when the root is not at $0$). In that case, the labels of vertices are indeed redundant and once they are dropped we get back to the edge-labeled trees used by Elizalde \cite{Eli}. However, in our situation with general multiset $\M$, it makes our later constructions of bijections easier by labeling vertices as well. 
\end{remark}

\begin{figure}[htb]
\begin{tikzpicture}[scale=0.4, decoration={markings, mark= at position 0.5 with {\arrow{stealth}}}
]
\draw[postaction={decorate}] (18,10) -- (18,12);
\draw[postaction={decorate}] (18,12) -- (18,14);
\draw[postaction={decorate}] (18,14) -- (20,16);
\draw[postaction={decorate}] (22,14) -- (20,16);
\draw[postaction={decorate}] (20,16) -- (24,20);
\draw[postaction={decorate}] (24,16) -- (24,20);
\draw[postaction={decorate}] (24,14) -- (24,16);
\draw[postaction={decorate}] (28,16) -- (24,20);
\draw[postaction={decorate}] (26,14) -- (28,16);
\draw[postaction={decorate}] (30,14) -- (28,16);
\draw[postaction={decorate}] (28,12) -- (30,14);
\draw[postaction={decorate}] (32,12) -- (30,14);
\draw[postaction={decorate}] (32,10) -- (32,12);

\node at (18,10) {$\bullet$};
\node at (18,12) {$\bullet$};
\node at (18,14) {$\bullet$};
\node at (20,16) {$\bullet$};
\node at (22,14) {$\bullet$};
\node at (24,20) {$\bullet$};
\node at (24,16) {$\bullet$};
\node at (24,14) {$\bullet$};
\node at (28,16) {$\bullet$};
\node at (26,14) {$\bullet$};
\node at (30,14) {$\bullet$};
\node at (28,12) {$\bullet$};
\node at (32,12) {$\bullet$};
\node at (32,10) {$\bullet$};

\node at (21.2,18) {$7$};
\node at (24.5,18) {$7$};
\node at (26.7,18) {$7$};
\node at (17.6,9.5) {\blue{$s_1$}};
\node at (18.5,11) {$1$};
\node at (17.4,12) {\blue{$1$}};
\node at (18.5,13) {$2$};
\node at (17.4,14) {\blue{$9$}};
\node at (18.6,15.4) {$8$};
\node at (19.5,16) {\blue{$6$}};
\node at (21.2,15.4) {$6$};
\node at (21.8,13.5) {\blue{$s_6$}};
\node at (24,20.6) {\blue{$0$}};
\node at (23.5,16) {\blue{$7$}};
\node at (24.5,15) {$4$};
\node at (23.5,14) {\blue{$4$}};
\node at (27.4,16) {\blue{$8$}};
\node at (26,13.4) {\blue{$10$}};
\node at (29.2,14) {\blue{$11$}};
\node at (27.4,14.9) {$9$};
\node at (29.4,15.3) {$9$};
\node at (29.3,12.8) {$3$};
\node at (31.5,13.2) {$3$};
\node at (27.5,11.8) {\blue{$2$}};
\node at (31.5,11.8) {\blue{$3$}};
\node at (32.5,10.9) {$5$};
\node at (31.5,9.8) {\blue{$5$}};

\end{tikzpicture}
\caption{The VE-labeled tree $T$ corresponding to $\phi(T)=78212867447993355397$} 
\label{fig:VE-tree}
\end{figure}
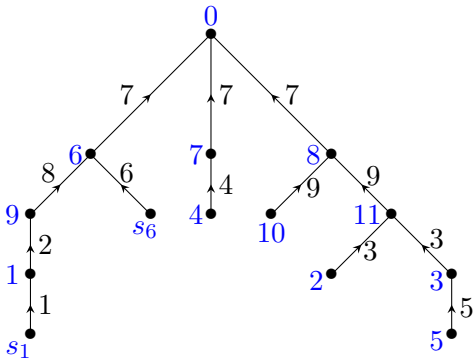

\begin{Def}\label{def:r-coding}
Given a tree $T\in \T_{\M}$ whose root is labeled $r$, we call the correspondence between edge labels and vertex labels the {\em $r$-coding} of the multiset $\M$. More generally, any graph whose edge labels and vertex labels are linked in the same way as $r$-coding is said to be {\em consistent} with $r$-coding.
\end{Def} 

This $r$-coding can be described explicitly. Namely, we first line up integers in $\M$ increasingly from left to right (integers of the same value are distinguished by the subscripts), and encode them one-by-one as $0,1,\ldots,r-1,r+1,\ldots,M-n$, skipping the first copy of each integer (including those singletons), then we encode the singleton $i$ by $s_i$. For the multiset associated with the tree in Fig.~\ref{fig:VE-tree}, its $0$-coding is presented in Table~\ref{coding}. We usually refer to $0$-coding as the standard coding of the multiset $\M$ and denote it as $c$, while the $r$-coding with $r\ge 1$ is said to be shifted and denoted as $c_r$. 
\begin{Def}\label{def:congruent}
Two vertices having the same preimage under the map $c_r$ are said to be {\em congruent with respect to $r$-coding}.
\end{Def} 

For example in Table~\ref{coding}, vertices labeled $6,7,8$ are all congruent with respect to $0$-coding.
\begin{table}[ht]
\centering
\begin{tabular}{cVc|c|c|c|c|c|c|c|c|c|c|c|c|c|c|c|c|c|c|c}
$i$ & $1$ & $2_1$ & $2_2$ & $3_1$ & $3_2$ & $3_3$ & $4_1$ & $4_2$ & $5_1$ & $5_2$ & $6$ & $7_1$ & $7_2$ & $7_3$ & $7_4$ & $8_1$ & $8_2$ & $9_1$ & $9_2$ & $9_3$\\
\Xhline{2pt}
$c(i)$ & $s_1$ & & $1$ & & $2$ & $3$ & & $4$ & & $5$ & $s_6$ & & $6$ & $7$ & $8$ & & $9$ & & $10$ & $11$
\end{tabular}
\vspace{3mm}
\caption{The $0$-coding of the multiset $\M=\{1,2^2,3^3,4^2,5^2,6,7^4,8^2,9^3\}$.}
\label{coding}
\end{table}

According to Definition~\ref{def:VE-tree}, each integer in $[M-n]_0$ could be used as the root label. Aside from $0$, these are precisely the images of the non-first copies of non-singletons in $\M$, under the standard coding function $c$. This observation gives rise to the following definition, which defines the object in the title of this paper.

\begin{Def}\label{def:rooted qStir}
Given any multiset $\M=\{1^{m_1},\ldots,n^{m_n}\}$ with $M=m_1+\cdots+m_n$, we denote $\R_{\M}$ the set of pairs $(\pi,k_j)$, where $\pi\in\bQ_{\M}$ and either $k_j=0$, or $k_j$ is the $j$-th copy of $k$ for certain $1\le k\le n$ and $2\le j\le m_k$. We call the pair $(\pi,k_j)$ a {\em quasi-Stirling multipermutation of $\M$ rooted at} $k_j$, and it can be succinctly represented by underlining the $j$-th (counting from left to right) copy of $k$ in $\pi$. The special case of $k_j=0$ can be thought of as either unrooted or rooted at $\pi_{M+1}=0$.
\end{Def}

\begin{remark}
It is clear that $\bQ_{\M}$ can be naturally embedded in $\R_{\M}$, corresponding to those pairs with $k_j=0$. Further note that due to the condition $2\le j\le m_k$, there are exactly $M-n+1$ (including the choice of $0$) choices for $k_j$, once the permutation $\pi$ is given. Thus we see immediately that
\begin{align}\label{R=bQ}
|\R_{\M}|=(M-n+1)|\bQ_{\M}|,
\end{align}
which is precisely the sum of coefficients for the left hand side of \eqref{genqStir-Euler}.  
\end{remark}

Note that underlining an entry of $\pi$ does not effect the numbers of descents, ascents and plateaus of $\pi$, so these three statistics extend to the pair $(\pi,k_j)$. In order to enumerate rooted quasi-Stirling multipermutations with respect to the statistics $\des$, $\asc$, and $\plat$ using VE-labeled trees, we recall the following definitions from \cite{Eli}.

Define the number of {\em cyclic descents} and {\em cyclic ascents} of a sequence of nonnegative integers $\pi=\pi_1\pi_2\ldots\pi_r$ to be

\begin{align*}
\cdes(\pi)=|\{i\in[r]:\pi_i>\pi_{i+1}\}|, \text{ and } \casc(\pi)=|\{i\in[r]:\pi_i<\pi_{i+1}\}|,
\end{align*} 
respectively, with the convention $\pi_{r+1}=\pi_1$ (not the usual convention $\pi_{r+1}=0$). So for example $\cdes(31221)=2$ while $\des(31221)=3$.

Let $T\in\T_{\M}$ and let $v$ be a vertex of $T$. Suppose the edges between $v$ and its children are labeled $a_1,a_2,\ldots,a_d$ from left to right, and the edge between $v$ and its parent, if any, is labeled as $b$. If $v$ is the root, define $\cdes(v)$ (resp.~$\casc(v)$) to be $\des(a_1\ldots a_d)$ (resp.~$\asc(a_1\ldots a_d)$). Otherwise, $\cdes(v)=\cdes(ba_1\ldots a_d)$ (resp.~$\casc(v)=\casc(ba_1\ldots a_d)$). Next, define the numbers of cyclic descents and cyclic ascents of $T$ to be
\begin{align*}
\cdes(T)=\sum_v \cdes(v), \text{ and } \casc(T)=\sum_v \casc(v),
\end{align*}
respectively, where both sums range over all the vertices $v$ of $T$. Finally, define $\lleaf(T)$ to be the number of integer-labeled leaves of $T$.

We are now ready for the main result of this section, which can be viewed as the first step towards proving \eqref{genqStir-Euler} bijectively. The reader is invited to use the tree in Fig.~\ref{fig:VE-tree} as one example of the bijection $\phi$ constructed below.

\begin{theorem}\label{thm:bijTtoR}
There exists a bijection $\phi: \T_{\M}\rightarrow \R_{\M}$, which induces a bijection between $\T_{\M}^0$ and $\bQ_{\M}$. Moreover, if $(\pi,k_j)=\phi(T)$, then we have
\begin{align}
\label{cdes=des}\cdes(T) &=\des(\pi),\\
\label{casc=asc}\casc(T) &=\asc(\pi),\\
\label{lleaf=plat}\lleaf(T) &=\plat(\pi).
\end{align}
\end{theorem}
\begin{proof}
Given a tree $T\in\T_{\M}$, we explain how to construct its image $(\pi,k_j)$ under $\phi$. We consider two cases according to the label of the root of $T$.
\begin{itemize}
	\item[Case I.\phantom{I}] $T\in\T_{\M}^0$, i.e., the root of $T$ is labeled $0$. In this case, let the image be $\phi(T)=(\pi,0)$, where the multipermutation $\pi$ is constructed as follows. We traverse the edges of $T$ by following a depth-first walk from left to right (also known as the {\em preorder traversal}). Namely, starting from the root, we go to the leftmost child and explore that branch recursively, return to the root, then move on to the next child, and so on (see \cite[Fig.~5-14]{Sta} for a pictorial illustration). Recording the edge labels as they are traversed produces a word $w$, which is not our final output $\pi$ yet. For each consecutively repeated pair $aa$ in $w$, we make the following adjustments accordingly.
	\begin{enumerate}[i]
		\item If this pair records labels from two edges that were tranversed consecutively (these two edges must be ending at the same vertex), replace $aa$ with $a$.
		\item If this pair records the same edge being tranversed twice consecutively, and this edge starts at a singleton-labeled leaf, replace $aa$ with $a$.
		\item If this pair records the same edge being tranversed twice consecutively, and this edge starts at an integer-labeled leaf, keep $aa$ as is.
	\end{enumerate}
	The new word we get after these adjustments is taken to be the multipermutation $\pi$. Recall that the tree $T$ has $M-\sum_{1\le i\le n}\chi(m_i>1)$ edges. The following calculation of the length of $\pi$ reflects the adjustments and verifies that $\pi$ indeed is a permutation of the multiset $\M$. The fact that $\pi$ is quasi-Stirling is guaranteed by the condition (3) in Definition \ref{def:VE-tree}.
	\begin{align*}
	& 2(M-\sum_{1\le i\le n}\chi(m_i>1))-\sum_{1\le i\le n}\chi(m_i=1)-\sum_{1\le i\le n}(m_i-2)\chi(m_i>1)\\
	& =2M-\sum_{1\le i\le n}m_i=M.
	\end{align*}
	\item[Case II.] The root of $T$ is labeled $r\in[M-n]$. Apply the same preorder traversal of the edges of $T$ as in case I to get the multipermutation $\pi$. Next, set $k_j=c^{-1}(r)$, the preimage of $r$ under the standard coding function $c$ of $\M$. This gives us the image $\phi(T)=(\pi,k_j)$.
\end{itemize}

Conversely, if we are given a rooted quasi-Stirling multipermutation $(\pi,k_j)$, we first reverse the tree traversal process to get the edge labeled tree $T$ from $\pi$, then use the standard coding function to get the root label $c(k_j)$. This root label, together with the edge labels of $T$, are sufficient for us to deduce the remaining vertex labels for $T$ (see Remark~\ref{rmk:conversion}). Hence $\phi$ is seen to be a bijection.

Next, to verify \eqref{cdes=des} and \eqref{casc=asc}, we carry out a case-by-case discussion on the types of descents (resp.~ascents) appearing in $\pi$, analogous to the proof of Lemma~2.1 in \cite{Eli}. The details are omitted.

Finally, \eqref{lleaf=plat} follows from the discussion of three cases i, ii, iii of pair $aa$ in the transition from $w$ to $\pi$ in Case I above, since the only situation that a plateau is preserved as we adjust $w$ to get $\pi$, is the case iii, which happens exactly when an integer-labeled leaf is traversed.
\end{proof}

\section{Bijective proofs of Theorems~\ref{thm:YYHZ} and \ref{thm:YZ}}\label{sect:thm1-2}

In this section, we give a new bijective proof of Theorem~\ref{thm:YYHZ}. This approach is also applicable to \eqref{numerator:genqStir}, giving us a unified treatment of both Theorems~\ref{thm:YYHZ} and \ref{thm:YZ}.

We begin by analyzing the right hand side of \eqref{genqStir-Euler}. For a vector $\mathbf{a}=(a_1,a_2,\ldots,a_l)\in \N^l$ consisting of $l$ nonnegative integers, we define the following two statistics:
\begin{align*}
|\mathbf{a}| &=a_1+a_2+\cdots+a_l,\\
|\mathbf{a}|_{\scaleto{0}{3.5pt}} &= |\{1\le i\le l:a_i=0\}|.
\end{align*}
Denoting $k:=M-n+1$ in the right hand side of \eqref{genqStir-Euler}, we have
\begin{align}
n![u^n](A(x,y,u)-1+z)^k &= \sum_{\mathbf{a}\in\N^{k},\:|\mathbf{a}|=n}\binom{n}{a_1,\ldots,a_k}z^{|\mathbf{a}|_{\scaleto{0}{3pt}}}\prod_{1\le i\le k,\:a_i>0}A_{a_i}(x,y) \nonumber \\
&=\sum_{\Pi\in\B_{n,k}}x^{\des(\Pi)}y^{\asc(\Pi)}z^{\emp(\Pi)},\label{gf:Bnk}
\end{align}
where $\B_{n,k}$ is the set of partitions $\Pi$ of $[n]$ into $k$ (possibly empty) blocks, such that each block itself is written as a permutation of the integers it contains. Moreover, the permutation statistics $\des$, $\asc$, and $\plat$ naturally extend to $\B_{n,k}$. Namely, for $\Pi=(\pi^{(0)},\ldots,\pi^{(k-1)})$, we let
$$\des(\Pi)=\sum_{i=0}^{k-1} \des(\pi^{(i)}),\; \asc(\Pi)=\sum_{i=0}^{k-1} \asc(\pi^{(i)}),\; \emp(\Pi)=\sum_{i=0}^{k-1} \plat(\pi^{(i)}).$$
Recalling the convention that only the empty permutation $\epsilon$ has one plateau, we see that $\emp(\Pi)$ is effectively the number of empty blocks in $\Pi$.

Combining \eqref{R=bQ}, Theorem~\ref{thm:bijTtoR}, and \eqref{gf:Bnk}, it is clear that Theorem~\ref{thm:YYHZ} is equivalent to the following identity:

\begin{align}\label{triplestat:T=B}
\sum_{T\in\T_{\M}}x^{\cdes(T)}y^{\casc(T)}z^{\lleaf(T)}=\sum_{\Pi\in\B_{n,k}}x^{\des(\Pi)}y^{\asc(\Pi)}z^{\emp(\Pi)}.
\end{align}

Ideally, one would expect a bijection from $\T_{\M}$ to $\B_{n,k}$ which transforms the triple statistics $(\cdes,\casc,\lleaf)$ over trees to $(\des,\asc,\emp)$ over partitions. This is unfortunately not the case with our bijection $\Psi$ constructed in the next theorem. For instance, the tree $T$ in Fig.~\ref{fig:VE-tree} has $\cdes(T)=8$, while its image $\Psi(T)=\Pi$ has $\des(\Pi)=9$. Nonetheless, this bijection works well when we consider trees and partitions in their equivalence classes, not individually.

Two VE-labeled trees $T$ and $T'$ in $\T_{\M}$ are said to be equivalent, denoted as $T\sim T'$, if for each vertex label $0\le i\le k-1=M-n$, the (labeled) edges ending at $i$ in $T'$ are just rearrangements of the edges ending at $i$ in $T$. All trees equivalent to a given tree $T$ form an equivalence class, denoted as $[T]$. Analogously, two partitions $\Pi,\Pi'\in\B_{n,k}$ are said to be equivalent, if the $i$-th block (written as a permutation) in $\Pi'$ is a rearrangement of the $i$-th block in $\Pi$, for $0\le i\le k-1$. The equivalence class containing $\Pi$ is denoted as $[\Pi]$.

\begin{theorem}\label{thm:FL}
Let $\M=\{1^{m_1},\ldots,n^{m_n}\}$ with $M=m_1+\cdots+m_n$ and $k=M-n+1$. There is a bijection $\Psi:\T_{\M}\rightarrow \B_{n,k}$, such that if $\Pi=\Psi(T)$, then we have
\begin{align}\label{eq:triple stat equiv}
\sum_{T'\in[T]}x^{\cdes(T')}y^{\casc(T')}z^{\lleaf(T')}=\sum_{\Pi'\in[\Pi]}x^{\des(\Pi')}y^{\asc(\Pi')}z^{\emp(\Pi')}.
\end{align}
Consequently, equations \eqref{triplestat:T=B} and \eqref{genqStir-Euler} hold in turn.
\end{theorem}

As it turns out, the construction of the bijection $\Psi$ is irrelevant to either the orders between sibling edges of the trees in $\T_{\M}$, or the orders between integers inside the same block of the partitions in $\B_{n,k}$. The proof of Theorem~\ref{thm:FL} thus hinges on its unordered version. We make this precise by first giving the following two definitions.
\begin{Def}
For any given multiset $\M$, let $\UT_{\M}$ denote the set of {\em unordered VE-labeled trees} over $\M$. These are trees satisfying all conditions (1)--(4) in Definition~\ref{def:VE-tree}, except that we ignore the orders between sibling edges. Similarly, let $\UB_{n,k}$ denote the set of usual set partitions of $[n]$ into $k$ (possibly empty) blocks, i.e., each block is viewed as a subset, not a permutation as in $\B_{n,k}$.
\end{Def}

\begin{Def}
For any given multiset $\M=\{1^{m_1},\ldots,n^{m_n}\}$, let $\G_{\M}$ denote the set of {\em regular graphs} over $\M$. These are directed and VE-labeled plane graphs satisfying all the labeling conditions (1)--(4) in Definition~\ref{def:VE-tree}, and vertex $0$ has outdegree $0$, while all other vertices have outdegree $1$. The unordered (i.e., ignoring the orders between sibling edges) regular graphs over $\M$ form a set denoted as $\UG_{\M}$.
\end{Def}

\begin{remark}
Note that $\T_{\M}\cap\G_{\M}=\T_{\M}^0$. Moreover, a key feature of the regular graphs over $\M$, is that they are consistent with the $0$-coding of $\M$. Therefore, for the sake of simplicity, we shall only label the vertices when we draw a regular graph (such as the graph in Fig.~\ref{fig:psi2} and the third graph in Fig.~\ref{fig:unordered VE-tree}), as long as the underlying multiset $\M$ is given.
\end{remark}

\begin{theorem}\label{thm:unorder}
Let $\M=\{1^{m_1},\ldots,n^{m_n}\}$ with $M=m_1+\cdots+m_n$ and $k=M-n+1$. There is a three-way correspondence
\begin{align}
\UT_{\M}\stackrel{\psi_1}{\longrightarrow}\UG_{\M}\stackrel{\psi_2}{\longrightarrow}\UB_{n,k},
\end{align}
where both $\psi_1$ and $\psi_2$ are bijections. Moreover, suppose $T\in\UT_{\M}$, $G=\psi_1(T)\in\UG_{\M}$, and $\Pi=\psi_2(G)\in\UB_{n,k}$, then for each $i\in[M-n]_0$, the following three sets are equinumerous:
\begin{enumerate}
	\item the edges with distinct labels ending at vertex labeled $i$ in $T$;
	\item the edges with distinct labels ending at vertex labeled $i$ in $G$;
	\item the integers contained in the block $\pi^{(i)}$ of $\Pi$.
\end{enumerate}
\end{theorem}
\begin{proof}
We start with the easier map $\psi_2:\UG_{\M}\rightarrow \UB_{n,k}$. For any function, say $f:A\rightarrow B$, with $A$ and $B$ being finite sets, there are two natural ways of representing $f$, other than listing out $f(i)$ for each $i\in A$. The first way is to draw the graph of $f$, say $G_f$, which is a directed graph with vertex set $A\cup f(A)$ and directed edges $i\rightarrow f(i)$. The second way is to write out all the preimages $f^{-1}(j)$ for each $j\in B$, as a set partition, say $\Pi_f$, of $A$ into $|B|$ blocks. For our purpose, the function playing this pivotal role is the parent function $$p=p_G:[M-n]\cup S_{\M}\rightarrow [M-n]_0$$ associated with any given regular graph $G$, which sends every vertex labeled either as integers from $[M-n]$ or as singletons from $S_{\M}$, to its uniquely found parent (since $G$ is regular, every nonzero vertex has outdegree $1$) whose label is from $[M-n]_0$. Now we can define the map $\psi_2$ as the composition of the following three maps. Take any $G\in\UG_{\M}$, we have
$$G\rightarrow p\rightarrow \Pi_p\rightarrow \Pi:=\psi_2(G),$$
where the first map sends $G$ to its associated parent function $p$, and the second map represents $p$ as a set partition $\Pi_p$, the third map then uses the $0$-coding of $\M$ to rewrite the integers contained in each block of $\Pi_p$ as their preimages under $c$ (repeated edge labels written only once), giving us a unique partition $\Pi$ of $[n]$ into $k$ blocks. An example of the map $\psi_2$ showing all three intermediate maps can be found in Fig.~\ref{fig:psi2}, where empty blocks are denoted by $\epsilon$ and blocks are separated by $\slash$. Since each step is invertible, $\psi_2$ is indeed a bijection. The equinumerousity between sets (2) and (3) should be clear from the construction of $\psi_2$.

\begin{figure}[htb]
\begin{tikzpicture}[scale=0.45, decoration={markings, mark= at position 0.55 with {\arrow{stealth}}}
]
\draw[postaction={decorate}] (0,0) -- (1.5,2);
\draw[postaction={decorate}] (3,0) -- (1.5,2);
\draw[postaction={decorate}] (5,0) -- (5,2);
\draw (4.9,2) edge[postaction={decorate},out=140,in=40,distance=20mm] (5.1,2);

\node at (0,0) {$\bullet$};
\node at (1.5,2) {$\bullet$};
\node at (3,0) {$\bullet$};
\node at (5,0) {$\bullet$};
\node at (5,2) {$\bullet$};

\node at (0,-.7) {$1$};
\node at (3,-.7) {$2$};
\node at (1.5,2.7) {$0$};
\node at (5,-.7) {$s_2$};
\node at (5.5,1.9) {$3$};

\draw (7.5,1) -- (13,1);
\draw (9,2) -- (9,0);
\node at (8,1.5) {$i$};
\node at (8,.5) {$p(i)$};
\node at (9.5,1.5) {$1$};
\node at (9.5,.5) {$0$};
\node at (10.5,1.5) {$2$};
\node at (10.5,.5) {$0$};
\node at (11.5,1.5) {$3$};
\node at (11.5,.5) {$3$};
\node at (12.5,1.5) {$s_2$};
\node at (12.5,.5) {$3$};

\node at (19,1) {$1,2~\slash~\epsilon~\slash~\epsilon~\slash~s_2,3$};

\node at (26.5,1) {$1~\slash~\epsilon~\slash~\epsilon~\slash~2,3$};

\node at (-1.5,1) {$G=$};
\node at (6.5,1) {$\rightarrow$};
\node at (14.5,1) {$\rightarrow$};
\node at (23,1) {$\rightarrow$};
\node at (31,1) {$=\psi_2(G)$};
\end{tikzpicture}
\caption{The transformation from $G$ to $\psi_2(G)$ with the given multiset $\M=\{1^3,2,3^2\}$}
\label{fig:psi2}
\end{figure}
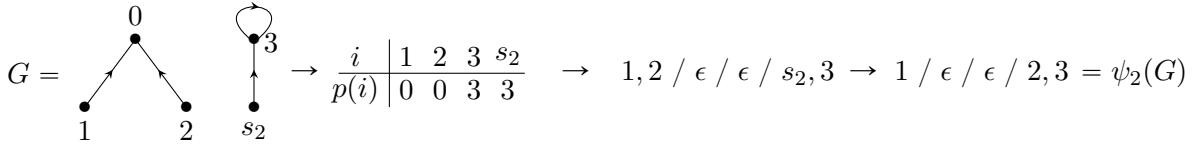

Next, we proceed to construct $\psi_1$. Given a tree $T\in\UT_{\M}$, if its root is labeled $0$, then it is already regular, in which case the map $\psi_1$ is the identity map, i.e., $\psi_1(T):=T$. Otherwise, suppose the root of $T$ is labeled $r$, for certain $1\le r\le M-n$. We transform $T$ to a regular graph $G$, whose features in contrast with $T$ are summarized in the following table. For each singleton vertex, say $s_i\in S_{\M}$, it is fixed throughout the whole construction of $G$, meaning that the parent of $s_i$ in $T$ remains the parent of $s_i$ in $G$. For a non-singleton vertex, its parents in $T$ and in $G$ may or may not be the same. The details are contained in the following two main steps. 
\begin{table}[htb]
\centering
\begin{tabular}{cVc|c|c}
 & consistent with & outdegree of vertex $0$ & outdegree of vertex $r$\\
\Xhline{2pt}
$T$ & $r$-coding & $1$ & $0$\\
$G$ & $0$-coding & $0$ & $1$
\end{tabular}
\vspace{3mm}
\caption{The comparison between $T$ and $\psi_1(T)=G$}
\label{T and G}
\end{table}

\begin{itemize}
	\item[Step 1)] In this step, we construct an intermediate graph $\tilde{G}$. The idea is to choose a unique representative, called the {\em anchor}, from each congruence class (see Definition~\ref{def:congruent}) of vertices in $T$ with respect to $r$-coding. Going from $T$ to $\tilde{G}$, the anchor vertices are all fixed, while other vertices may have to change their parents. More precisely, let
	$$A_T:=\{i\in[M-n]_0: \text{if $j$ is congruent to $i$ with respect to $r$-coding, then $i\le j$}\}\cup S_{\M}$$
	be the set of anchor vertices of $T$. Note that in particular, the root vertex $r$ is always an anchor (since there exists no other vertices that are congruent to $r$), so is the vertex $0$ (since it is the smallest label in value). Moreover, each congruence class with respect to $r$-coding contains exactly one anchor, and there are $n+1$ anchors in total. Now let $\tilde{G}$ be the unique graph with the same vertex set as $T$, such that the following conditions are satisfied.
	\begin{enumerate}[(i)]
	\item $p_{\tilde{G}}(s_i)=p_T(s_i)$ for each singleton $s_i\in S_{\M}$, and $p_{\tilde{G}}(0)=p_T(0)$.
	\item The vertices of $\tilde{G}$ having outdegree $0$ are precisely those vertices congruent to $r$ (including $r$ itself) with respect to $0$-coding. Note that they must be consecutively labeled, say as $[\ell,\ell+q]:=\{\ell,\ell+1,\ldots,\ell+q\}$.
	\item For the remaining integer-labeled vertex $i\in[M-n]\setminus [\ell,\ell+q]$, we have
	\begin{align*} 
	p_{\tilde{G}}(i)=p_T(j),
	\end{align*}
	where $j$ is the unique integer such that $j\in A_T$ and $j$ is congruent to $i$ with respect to $0$-coding. 
	\end{enumerate}
	Conversely, to go from $\tilde{G}$ back to $T$, we first realize that $r=r(T)$ must be a vertex from $[\ell,\ell+q]$, i.e., those vertices in $\tilde{G}$ having outdegree $0$. Knowing this is enough to determine the congruence classes with respect to $r$-coding for all vertices in $[M-n]\setminus [\ell,\ell+q]$. Next, use the largest labeled vertices in each congruence class (except $[\ell,\ell+q]$) with respect to $0$-coding as the anchor vertices, keep their parents unchanged while adjusting the parents of non-anchor vertices so that vertices within the same congruence class ($r$-coding) become siblings. Now observe that among all vertices in $[\ell,\ell+q]$, the one that should be the root $r$ of $T$ is exactly the one currently having $\ell-1$ as its descendant. The final step is to make each vertex from $[\ell,\ell+q]\setminus\{r\}$ a sibling vertex of $\ell-1$, yielding our desired tree $T$. Hence we see that step 1) is indeed invertible.
	\item[Step 2)] In this second step, we make some further adjustments to $\tilde{G}$ and derive $G:=\psi_1(T)$. Recall that if $r=c(a_j)$ is the root of $T$, then there are $m_a-1$ vertices in $\tilde{G}$ having outdegree $0$, as a result of condition (ii) from Step 1). Suppose $t$ is the one that has $0$ as its descendant, and let $0=v_0\rightarrow v_1\rightarrow\cdots\rightarrow v_s=t$ be the path from $0$ to $t$ in $\tilde{G}$. Now we 
	\begin{enumerate}[i)]
	\item find and relabel the right-to-left minima of the word $v_0v_1\cdots v_s$ as $$u_0=v_0=0<u_1<\cdots<u_j=v_s=t;$$
	\item delete the edge $u_i\rightarrow p(u_i)$ for each $0\le i<j$;
	\item add the edge $u_i\rightarrow p(u_{i-1})$ for each $0<i\le j$. 
	\end{enumerate}
	For example, the path $0\rightarrow 3\rightarrow 1\rightarrow 5\rightarrow 11\rightarrow 12\rightarrow 8$ becomes $0$, the cycle $3\rightarrow 1\rightarrow 3$, the loop $5\rightarrow 5$, and the cycle $11\rightarrow 12\rightarrow 8\rightarrow 11$. This 3-step operation probably remind the reader of Foata's first fundamental transformation \cite[Chap.~10.2]{Lot}. Finally, note that each vertex $v_0,v_1,\ldots,v_s$ along the original path is contained in a different congruence class with respect to $0$-coding, so there is a unique way to adjust accordingly the parents of those vertices congruent to certain $u_i$, so as to produce a regular graph that we denote as $G$. Just like Foata's first fundamental transformation is a bijection, it should be clear how to reverse this step 2) and uniquely recover $\tilde{G}$ from any given regular graph $G$.
\end{itemize}

In conclusion, the map $\psi_1:\UT_{\M}\rightarrow\UG_{\M}$ consisting of the two steps 1) and 2) above is a bijection that ensures the equinumerousity between the sets (1) and (2). The proof is now completed.
\end{proof}

\begin{example}
Let $T$ be a tree in $\UT_{\M}$ with $\M=\{1^2,2^3,3^2,4,5^3,6^2,7^4,8\}$ as shown in Fig.~\ref{fig:unordered VE-tree}, the $0$-coding and $5$-coding of the multiset $\M$ can be found in Table~\ref{$0$-coding and $5$-coding}, and the set of anchor vertices corresponding to the tree $T$ is given by  $$A_T=\{0,1,3,4,5,7,8,s_4,s_8\}.$$
By applying the map $\psi_1$ and $\psi_2$, we get a set partition $\Pi\in\UB_{8,11}$ as the final output in Fig.~\ref{fig:unordered VE-tree}.
\end{example}

\begin{table}[h]
\centering
\begin{tabular}{cVc|c|c|c|c|c|c|c|c|c|c|c|c|c|c|c|c|c}
$i$ & $1_1$ & $1_2$ & $2_1$ & $2_2$ & $2_3$ & $3_1$ & $3_2$ & $4$ & $5_1$ & $5_2$ & $5_3$ & $6_1$ & $6_2$ & $7_1$ & $7_2$ & $7_3$ & $7_4$ & $8$\\
\Xhline{2pt}
$c(i)$ & & $1$ & & $2$ & $3$ & & $4$ & $s_4$ & & $5$ & $6$ & & $7$ & & $8$ & $9$ & $10$ & $s_8$\\
\Xhline{0.5pt}
$c_5(i)$ & & $0$ & & $1$ & $2$ & & $3$ & $s_4$ & & $4$ & $6$ & & $7$ & & $8$ & $9$ & $10$ & $s_8$
\end{tabular}
\vspace{3mm}
\caption{The $0$-coding and $5$-coding of the multiset $\M=\{1^2,2^3,3^2,4,5^3,6^2,7^4,8\}$.}
\label{$0$-coding and $5$-coding}
\end{table}

\begin{figure}[htb]
\begin{tikzpicture}[scale=0.3, decoration={markings, mark= at position 0.5 with {\arrow{stealth}}}]

\draw[postaction={decorate}] (-7,10) -- (-6,12);
\draw[postaction={decorate}] (-5,10) -- (-6,12);
\draw[postaction={decorate}] (-6,12) -- (1,16);
\draw[postaction={decorate}] (-3,8) -- (-3,10);
\draw[postaction={decorate}] (-3,10) -- (-1,12);
\draw[postaction={decorate}] (0,8) -- (1,10);
\draw[postaction={decorate}] (2,8) -- (1,10);
\draw[postaction={decorate}] (1,10) -- (-1,12);
\draw[postaction={decorate}] (-1,12) -- (1,16);
\draw[postaction={decorate}] (3,12) -- (1,16);
\draw[postaction={decorate}] (8,10) -- (8,12);
\draw[postaction={decorate}] (8,12) -- (1,16);

\node at (-7,10) {$\bullet$};
\node at (-5,10) {$\bullet$};
\node at (-6,12) {$\bullet$};
\node at (-3,8) {$\bullet$};
\node at (0,8) {$\bullet$};
\node at (2,8) {$\bullet$};
\node at (-3,10) {$\bullet$};
\node at (1,10) {$\bullet$};
\node at (-1,12) {$\bullet$};
\node at (3,12) {$\bullet$};
\node at (8,10) {$\bullet$};
\node at (8,12) {$\bullet$};
\node at (1,16) {$\bullet$};

\node at (-7.5,9.3) {$1$};
\node at (-5,9.3) {$2$};
\node at (-6.5,12) {$7$};
\node at (1,17) {$5$};
\node at (-1.8,12) {$8$};
\node at (-3.5,10) {$4$};
\node at (-3.8,8) {$s_8$};
\node at (-0.7,8) {$3$};
\node at (2.8,8) {$s_4$};
\node at (1.7,10) {$6$};
\node at (3.7,12) {$9$};
\node at (8.7,12) {$10$};
\node at (8.7,10) {$0$};

\node at (-7.2,11) {$2$};
\node at (-5,11) {$2$};
\node at (-3,14.5) {$6$};
\node at (-0.3,14.5) {$7$};
\node at (-2.7,11) {$5$};
\node at (-2.5,9) {$8$};
\node at (0,9) {$3$};
\node at (2.1,9) {$4$};
\node at (0.7,11) {$5$};
\node at (2.1,14.5) {$7$};
\node at (4.7,14.5) {$7$};
\node at (7.4,11) {$1$};

\node at (1,6) {$T$};
\node at (14,12) {$\rightarrow$};

\draw[postaction={decorate}] (19,10) -- (19,12);
\draw[postaction={decorate}] (19,12) -- (22,16);
\draw[postaction={decorate}] (21,8) -- (21,10);
\draw[postaction={decorate}] (21,10) -- (21,12);
\draw[postaction={decorate}] (21,12) -- (22,16);
\draw[postaction={decorate}] (23,12) -- (22,16);
\draw[postaction={decorate}] (25,10) -- (25,12);
\draw[postaction={decorate}] (25,12) -- (22,16);

\node at (19,10) {$\bullet$};
\node at (19,12) {$\bullet$};
\node at (22,16) {$\bullet$};
\node at (21,8) {$\bullet$};
\node at (21,10) {$\bullet$};
\node at (21,12) {$\bullet$};
\node at (23,12) {$\bullet$};
\node at (25,10) {$\bullet$};
\node at (25,12) {$\bullet$};

\node at (18.4,10) {$1$};
\node at (18.4,12) {$7$};
\node at (22,17) {$5$};
\node at (20.2,8) {$s_8$};
\node at (20.4,10) {$4$};
\node at (20.4,12) {$8$};
\node at (23.6,12) {$9$};
\node at (25.6,10) {$0$};
\node at (25.8,12) {$10$};

\draw[postaction={decorate}] (28,12) -- (30,16);
\draw[postaction={decorate}] (30,12) -- (30,16);
\draw[postaction={decorate}] (32,12) -- (30,16);

\node at (28,12) {$\bullet$};
\node at (30,12) {$\bullet$};
\node at (32,12) {$\bullet$};
\node at (30,16) {$\bullet$};

\node at (27.4,12) {$2$};
\node at (30.6,12) {$3$};
\node at (32.8,12) {$s_4$};
\node at (30,17) {$6$};

\node at (26,6) {$\tilde{G}$};

\node at (26,4) {$\downarrow$};

\draw[postaction={decorate}] (19,-5) -- (19,-3);
\draw[postaction={decorate}] (19,-3) -- (22,1);
\draw[postaction={decorate}] (21,-7) -- (21,-5);
\draw[postaction={decorate}] (21,-5) -- (21,-3);
\draw[postaction={decorate}] (21,-3) -- (22,1);
\draw[postaction={decorate}] (23,-3) -- (22,1);
\draw[postaction={decorate}] (24,-7) -- (26,-5);
\draw[postaction={decorate}] (26,-7) -- (26,-5);
\draw[postaction={decorate}] (28,-7) -- (26,-5);
\draw[postaction={decorate}] (26,-5) -- (26,-3);
\draw[postaction={decorate}] (26,-3) -- (22,1);
\draw (22.1,0.7) edge[postaction={decorate},out=130,in=80,distance=20mm] (26,-3);

\node at (19,-5) {$\bullet$};
\node at (19,-3) {$\bullet$};
\node at (22,1) {$\bullet$};
\node at (21,-7) {$\bullet$};
\node at (21,-5) {$\bullet$};
\node at (21,-3) {$\bullet$};
\node at (23,-3) {$\bullet$};
\node at (24,-7) {$\bullet$};
\node at (26,-7) {$\bullet$};
\node at (28,-7) {$\bullet$};
\node at (26,-5) {$\bullet$};
\node at (26,-3) {$\bullet$};

\node at (18.4,-5) {$1$};
\node at (18.4,-3){$7$};
\node at (22,2) {$5$};
\node at (20.2,-7) {$s_8$};
\node at (20.4,-5) {$4$};
\node at (20.4,-3) {$8$};
\node at (23.6,-3) {$9$};
\node at (23.4,-7) {$2$};
\node at (26.7,-7) {$3$};
\node at (28.8,-7) {$s_4$};
\node at (26.7,-5) {$6$};
\node at (26.7,-3) {$10$};

\node at (30,1) {$\bullet$};
\node at (30,2) {$0$};

\node at (26,-9) {$G=\psi_1(T)$};

\node at (15,-3) {$\leftarrow$};

\node at (1,-3) {$\Pi_p=\epsilon\slash\epsilon\slash\epsilon\slash\epsilon\slash s_8\slash7,8,9,10\slash2,3,s_4\slash1\slash4\slash\epsilon\slash5,6$};

\node at (1,-8) {$\downarrow$};

\node at (1,-12) {$\Pi=\psi_2(G)=\epsilon\slash\epsilon\slash\epsilon\slash\epsilon\slash8\slash6,7\slash2,4\slash1\slash3\slash\epsilon\slash5$};
\end{tikzpicture}
\caption{An example of the bijection $\psi_1$ and $\psi_2$}
\label{fig:unordered VE-tree}
\end{figure}
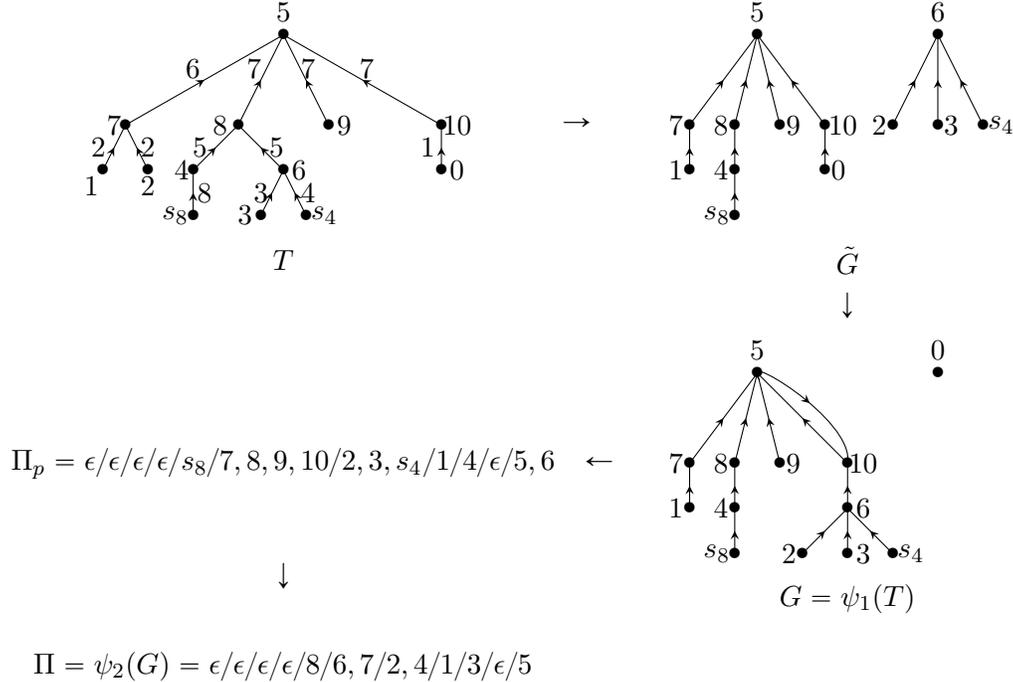

The composition of the two bijections $\psi_1$ and $\psi_2$ constructed in Theorem~\ref{thm:unorder} is a bijection from $\UT_{\M}$ to $\UB_{n,k}$, such that for $T\in\UT_{\M}$ and each $i\in[M-n]_0$, there are as many edges with distinct labels ending at vertex $i$ in $T$ as integers contained in the block $\pi^{(i)}$ of $\Pi:=\psi_2(\psi_1(T))$. Consequently, permuting the edges ending at $i$ corresponds uniquely to permuting the integers contained in the block $\pi^{(i)}$. In other words, we can lift the composition $\psi_2\circ\psi_1$ to a bijection $\Psi:\T_{\M}\rightarrow \B_{n,k}$. Namely, for an ordered tree $T\in\T_{\M}$, we ``forget'' the relative orders between sibling edges to obtain the unique unordered tree, say $\overline{T}\in\UT_{\M}$, map it to the partition $\overline{\Pi}:=\psi_2(\psi_1(\overline{T}))\in\UB_{n,k}$, then permute the integers inside each block $\pi^{(i)}$ of $\overline{\Pi}$ so that the word obtained is order isomorphic to the word consisted of the distinct labels of the edges ending at vertex $i$ of $T$. This ordered partition is the image $\Pi=\Psi(T)\in\B_{n,k}$. The mapping $\Psi$ defined this way is clearly a bijection and will be used to prove Theorem~\ref{thm:FL}. We still need to explain why the triple statistics $(\cdes,\casc,\lleaf)$ are transformed to $(\des,\asc,\emp)$. To this end, we first show the following lemma. 

\begin{lemma}
For $n\ge 1$, let $A_n^{(c)}(x,y):=\sum_{\pi\in\s_n}x^{\cdes(\pi)}y^{\casc(\pi)}$ be the bivariate cyclic Eulerian polynomial, then we have:
\begin{align}\label{eq:cdes=ndes}
A_n^{(c)}(x,y)=nA_{n-1}(x,y).
\end{align}
\end{lemma}
\begin{proof}
Note that for a fixed permutation $\pi\in\s_n$, we always have $\cdes(\pi)+\casc(\pi)=n$ and $\des(\pi)+\asc(\pi)=n+1$. Thus it suffices to show the identity after we set $y=1$ in \eqref{eq:cdes=ndes}. This univariate version is already known, see \cite[Coro.~1]{Ful} and \cite[Prop.~1.1]{Pet} for two proofs. 
\end{proof}

\begin{proof}[Proof of Theorem~\ref{thm:FL}]
We have already defined the bijection $\Psi$, which indeed maps the equivalence class $[T]$ to $[\Psi(T)]$. Moreover, note that $\lleaf(T)=\emp(\Psi(T))$, and both statistics $\lleaf$ and $\emp$ are constant on an equivalence class. 
Let $\sigma=(1~2\cdots n)$ be the $n$-cycle in $\s_n$, we see that for a given $\pi\in\s_n$, all permutations $\pi\sigma^i,i=0,1,\ldots,n-1$, have the same number of cyclic descents and cyclic ascents. Therefore, if we let $\pi$ run over all permutations in $\s_n$ with a predetermined first letter $\pi_1$, the generating function of the pair $(\cdes,\casc)$ is given by $A_n^{(c)}(x,y)/n$. Relying on this oberservation as well as \eqref{eq:cdes=ndes}, we have
\begin{align*}
\sum_{T'\in[T]}x^{\cdes(T')}y^{\casc(T')} =A_{n_r}(x,y) \prod_{n_i> 0,\:i\neq r}\frac{A_{n_i+1}^{(c)}(x,y)}{n_i+1} = \prod_{n_i>0}A_{n_i}(x,y) = \sum_{\Pi'\in[\Psi(T)]}x^{\des(\Pi')}y^{\asc(\Pi')},
\end{align*}
where $r$ is the root label of $T$, and $n_i$ is on one hand, the number of distinct labels of edges ending at vertex $i$ of $T$, and on the other hand, the number of integers contained in block $\pi^{(i)}$ of $\Psi(T)$. This proves \eqref{eq:triple stat equiv} and comletes the proof of Theorem~\ref{thm:FL}.
\end{proof}

We devote the rest of this section to the discussion on the proofs of Theorem~\ref{thm:YZ}. To deduce \eqref{numerator:genqStir} from \eqref{genqStir-Euler}, the quickest way after setting $y=z=1$ in \eqref{genqStir-Euler}, is to utilize the expression \eqref{gf:Eulerian poly} of the generating function $A(t,u)$ and follow the approach used by Elizalde \cite[Thm.~2.5]{Eli}. Alternatively, building on the combinatorial interpretation of \eqref{genqStir-Euler}, we supply here a bijective proof \`a la Gessel and Stanley \cite{GS}.

Let us first recall the notion of {\em barred permutations} (see e.g.~\cite{GS}). These are sequences of integers and bars ($/$) such that the integers form a word $w$ with distinct letters, and there is at least one bar in each descent of $w$. Now define $\overline{\B}_{n,k}$ to be the set of {\em barred partitions} of $[n]$ into $k$ blocks, where each block is written as a barred permutation. Note that each block $\pi^{(i)}=a_1a_2\cdots a_{n_i}$ provides $n_i+1$ spaces (inbetween $a_j$ and $a_{j+1}$, before $a_1$, and after $a_{n_i}$) where bars can be inserted, giving in total $n+k=n+(M-n+1)=M+1$ such spaces. To avoid confusion, we now use the curly brackets $\{\}$ to separate blocks and reserve the symbol $/$ for bars inside each block. For example, $(\{/3//14/\},\{///\},\{2///5//\})$ is a barred partition in $\overline{\B}_{5,3}$, while $(\{/3//14/\},\{///\},\{2///5\})$ is not, since in the third block there are no bars after the ending descent at $5$.

\begin{proof}[Bijective proof of Theorem~\ref{thm:YZ}]
Thanks to \eqref{genqStir-Euler} and \eqref{gf:Bnk}, we can interprete the right hand side of \eqref{numerator:genqStir} as:
\begin{align}\label{gf:bar over B}
\frac{(M-n+1)\overline{Q}_{\M}(t)}{(1-t)^{M+1}} = \frac{\sum_{\Pi\in\B_{n,k}}t^{\des(\Pi)}}{(1-t)^{M+1}} = \sum_{\Pi\in\overline{\B}_{n,k}}t^{\bar(\Pi)},
\end{align}
where $\bar(\Pi)$ is the total number of bars inserted in all blocks of $\Pi$. 

To connect with the left hand side of \eqref{numerator:genqStir}, for a fixed integer $m$ we enumerate barred partitions with $m$ bars in another way. Firstly, we determine how many bars are contained in each block. There are $k=M-n+1$ blocks and $m$ bars, so the number of different ways to insert bars into blocks is given by $\binom{M-n+m}{m}$. Once the bars are in position, it remains to place the integers $1,2,\ldots,n$. This step can be intuitively thought of as placing $n$ labeled balls (integers) into $m$ labeled boxes (bars to the immediate right), with balls in the same box aligned increasingly from left to right (since there must be at least one bar at each descent). In other words, the relative order between the balls inside each box is irrelevant. The number of ways to accomplish this second step is then given by $m^n$. Thus, the coefficient of $t^m$ in \eqref{gf:bar over B} is $\binom{M-n+m}{m}m^n$, as desired.
\end{proof}

\section{Partial gamma positivity and a proof of Theorem~\ref{thm:pargamma}}

The notion of gamma-positivity has attracted a considerable amount of interest recently, with various perspectives coming from enumerative combinatorics, enumerative geometry, as well as poset homology, see the survey by Athanasiadis \cite{Ath} and the references therein. A univariate polynomial $f(x)$ is said to be {\em $\gamma$-positive} if it has an expansion
$$f(x)=\sum_{k=0}^{\lfloor\frac{n}{2}\rfloor}\gamma_kx^k(1+x)^{n-2k}$$ with $\gamma_k\ge 0$. A bivariate polynomial $g(x,y)$ is said to be {\em homogeneous $\gamma$-positive}, if it can be expressed as
$$g(x,y)=\sum_{k=0}^{\lfloor\frac{n}{2}\rfloor}\gamma_k(xy)^k(x+y)^{n-2k}$$
with $\gamma_k\ge 0$. A well-known prototype of homogeneous $\gamma$-positive polynomial is the aforementioned bivariate Eulerian polynomial
\begin{align}\label{gamma of biEuler}
A_n(x,y)=\sum_{k=1}^{\lfloor\frac{n+1}{2}\rfloor}\gamma_{n,k}(xy)^k(x+y)^{n+1-2k}, \; n\ge 1.
\end{align}
Here the coefficient $\gamma_{n,k}$ is not only nonnegative, it has the following explicit combinatorial interpretation which was first derived by Foata and Strehl \cite{FS} via the well-known group action called ``valley-hopping''; see also \cite{SW,LZ,LMZ}. Recall that a {\em double descent} of $\pi\in\s_n$ is any index $1\le i\le n$ such that $\pi_{i-1}>\pi_i>\pi_{i+1}$ with the convention that $\pi_0=\pi_{n+1}=0$. Denoting $\ddes(\pi)$ the number of double descents of $\pi$, we have
\begin{align}\label{gamma-coef:biEuler}
\gamma_{n,k}=\#\{\pi\in\s_n:\des(\pi)=k,\ddes(\pi)=0\}.
\end{align}
It is worth mentioning that in a recent work of Sun \cite{Sun}, another kind of bivariate Eulerian polynomial was introduced and shown to enjoy similar but nonhomogeneous $\gamma$-positivity.

For trivariate polynomials, a notion that naturally extends homogeneous $\gamma$-positivity is {\em partial $\gamma$-positivity}, see \cite{SZ,LZ,MMY,LMZ} for recent work on several partial $\gamma$-positive polynomials. A trivariate polynomial $h(x,y,z)$ is called partial $\gamma$-positive if it can be expanded as $h(x,y,z)=\sum_i s_i(x,y)z^i$ with $s_i(x,y)$ being a homogeneous $\gamma$-positive polynomial for every $i$.

In our interpretation of the gamma coefficient $\gamma_{\M,i,j}$ in \eqref{pargamma-coef}, the statistics $\sdes$ and $\sddes$ are undefined. We now give their definitions. 
\begin{Def}\label{def:sdes}
Given a quasi-Stirling multipermutation $\pi=\pi_1\pi_2\cdots\pi_M\in\overline{\Q}_{\M}$, an index $i$, $1\le i\le M$, is said to be a {\em sibling descent} of $\pi$, if the following two conditions are satisfied:
\begin{enumerate}
	\item $\pi_i$ is the last copy among all entries with the same value;
	\item either $\pi_{i+1}$ is the first copy of its value and $\pi_i>\pi_{i+1}$ (type I), or $\pi_{i+1}$ is a non-first copy of its value (type II).
\end{enumerate}
Here we use again the convention $\pi_0=\pi_{M+1}=0$, so that $\pi_{M+1}$ is the second copy of $0$. An index $i$, $2\le i\le M$, is called a {\em double sibling descent} of $\pi$, if both $i-1$ and $i$ are sibling descents of $\pi$ and $i-1$ is of type I. The number of sibling descents (resp.~double sibling descents) of $\pi$ is denoted as $\sdes(\pi)$ (resp.~$\sddes(\pi)$).
\end{Def}

It is worth pointing out, that although our definitions of sibling descents and double sibling descents are a bit complicated, they do specialize to the classical statistics descents and double descents when the multiset $\M$ is taken to be $\{1,2,\ldots,n\}$. In this case, each entry appears once in the permutation so $\overline{\Q}_{\M}$ reduces to $\s_n$, and each sibling descent, except for the ending descent, is of type I, and is actually the usual descent (since condition (1) is now trivially true). I.e., for each $\pi\in\s_n$, $\sdes(\pi)=\des(\pi)$ and $\sddes(\pi)=\ddes(\pi)$. So we see \eqref{pargamma-coef} degenerates to \eqref{gamma-coef:biEuler}.

Before proving Theorem~\ref{thm:pargamma}, we make a quick observation and a remark. Combining \eqref{genqStir-Euler} with \eqref{gf:Bnk}, and noting that the product of two or more homogeneous $\gamma$-positive polynomials is still homogeneous $\gamma$-positive, we see immediately that the original conjecture of Lin, Ma and Zhang on the partial $\gamma$-positivity of $\overline{Q}_{\M}(x,y,z)$ holds true. It is the combinatorial meaning of the $\gamma$-coefficients that needs more effort to uncover. Recall that Yan, Huang and Yang \cite{YHY} also confirmed this conjecture and provided a combinatorial interpretation for the $\gamma$-coefficients. Their interpretation is in terms of statistics defined on certain vertex labeled trees, not directly on quasi-Stirling multipermutations. More precisely, the set of ordered (vertex) labeled trees used in \cite{YHY} is in simple bijection with the set of regular VE-labeled trees $\T_{\M}^0$ defined in this paper. Now let $\dcdes(T)$ be the number of double cyclic descents of the tree $T\in\T_{\M}^0$ (see \cite{YHY} for its definition), then the interpretation found by Yan et al. \cite[Thm.~3.2]{YHY} can be rephrased as
\begin{align}\label{pargamma:Yan}
\gamma_{\M,i,j}=\#\{T\in\T_{\M}^0:\lleaf(T)=i,\cdes(T)=j,\dcdes(T)=0\}.
\end{align}

Comparing \eqref{pargamma:Yan} with \eqref{pargamma-coef}, we get
\begin{corollary}
For any multiset $\M=\{1^{m_1},\ldots,n^{m_n}\}$ with $M=m_1+\cdots+m_n$, and $0\le i\le M-n$, $1\le j\le \lfloor(M+1-i)/2\rfloor$, the two sets
$$\{\pi\in\bQ_{\M}:\plat(\pi)=i,\sdes(\pi)=j,\sddes(\pi)=0\}$$
and
$$\{T\in\T_{\M}^0:\lleaf(T)=i,\cdes(T)=j,\dcdes(T)=0\}$$
are equinumerous.
\end{corollary}

Note that our bijection $\phi$ does restrict to a bijection from $\T_{\M}^0$ to $\overline{\Q}_{\M}$, but not to the subsets refined by the statistics. Therefore, it remains an interesting problem to find a direct bijection that proves the corollary above.

\begin{proof}[Proof of Theorem~\ref{thm:pargamma}]
Basing on the expression in \eqref{gf:Bnk} and Foata-Strehl's interpretation \eqref{gamma-coef:biEuler} for the $\gamma$-coefficients of $A_n(x,y)$, we see that $(M-n+1)\gamma_{\M,i,j}$, i.e., the coefficient of $z^i(xy)^j(x+y)^{M+1-i-2j}$ in $(M-n+1)\overline{Q}_{\M}(x,y,z)$, is precisely the cardinality of the set
\begin{align*}
\Gamma_{n,k,i,j}:=\{\Pi\in\B_{n,k}:\emp(\Pi)=i,\des(\Pi)=j,\ddes(\Pi)=0\},
\end{align*}
where $\ddes(\Pi)=\sum_{\ell}\ddes(\pi^{(\ell)})$, with the sum running over all blocks $\pi^{(\ell)}$ of $\Pi$. We trust the reader to verify the following fact.
\begin{Fact}
The composition $\phi\circ\Psi^{-1}:\B_{n,k}\rightarrow \R_{\M}$ is a bijection that sends the triple statistics $(\emp,\des,\ddes)$ over partitions from $\B_{n,k}$ to $(\plat,\sdes,\sddes)$ over rooted permutations from $\R_{\M}$.
\end{Fact}
Relying on this fact, we deduce that
\begin{align*}
\gamma_{\M,i,j} &= \frac{\# \Gamma_{n,k,i,j}}{M-n+1}= \frac{\# \{(\pi,r_t)\in\R_{\M}: \plat(\pi)=i,\sdes(\pi)=j,\sddes(\pi)=0\}}{M-n+1}\\
&= \#\{\pi\in\overline{\Q}_{\M}: \plat(\pi)=i,\sdes(\pi)=j,\sddes(\pi)=0\}.
\end{align*}
The second line uses the fact that the values of the three statistics $\plat,\sdes,\sddes$ are irrelevant to the root label $r_t$.
\end{proof}

\section*{Acknowledgement}
Both authors were supported by the National Natural Science Foundation of China grant 12171059.

\end{document}